\documentclass[12pt]{amsart}
\usepackage[all]{xy}
\usepackage{a4wide}
\usepackage{amssymb}
\usepackage{caption}
\usepackage{amsthm}
\usepackage{amsmath}
\usepackage{amscd,enumitem}
\usepackage{verbatim}
\usepackage{float}
\usepackage{color}
\usepackage{wasysym}
\usepackage[all]{xy}
\usepackage{hyperref}
\usepackage{cleveref}
\numberwithin{equation}{section}

\theoremstyle{plain}
\newtheorem{theorem}{Theorem}
\numberwithin{theorem}{section}
\newtheorem{corollary}[theorem]{Corollary}

\newtheorem{lemma}[theorem]{Lemma}
\newtheorem{proposition}[theorem]{Proposition}

\theoremstyle{definition}

\newtheorem*{theorem*}{Theorem}

\theoremstyle{remark}
\newtheorem*{remark}{Remark}

\newcommand{\Pod}{\mathrm{pod}}

\begin{document}


\title{Composition-theoretic series in partition theory}

\author{Robert Schneider and Andrew V. Sills}

\address{Department of Mathematical Sciences\newline
Michigan Technological University\newline
Houghton, Michigan 49931, U.S.A.}
\email{robertsc@mtu.edu}

\address{Department of Mathematical Sciences\newline
Georgia Southern University\newline
Statesboro, Georgia 30458, U.S.A.}
\email{asills@georgiasouthern.edu}


\begin{abstract} We use sums over  {integer compositions}  analogous to generating functions in  {partition theory}, to express certain partition enumeration functions as sums over compositions into parts that are $k$-gonal numbers; our proofs employ Ramanujan's theta functions. We explore applications to lacunary $q$-series, and to a new class of composition-theoretic Dirichlet series.  
\end{abstract}

\maketitle

\section{Introduction and statement of results}\label{section1}

\subsection{Partitions vs. compositions}
In this paper we use sums over  {integer compositions}  analogous to generating functions in  {partition theory}, to express certain partition enumeration functions as sums over compositions into parts that are $k$-gonal numbers, and other applications.

Let $\mathbb N$ denote the {\it natural numbers} (positive integers). Let $\mathcal P$ denote the set of {\it integer partitions}, unordered finite sums of natural numbers including the empty partition $\emptyset\in \mathcal P$ (see e.g. \cite{And}). For nonempty $\lambda\in \mathcal P$, we notate $\lambda=(\lambda_1, \lambda_2, \dots, \lambda_r)$,  $\lambda_1\geq \lambda_2\geq \dots \geq \lambda_r \geq 1$. Let $S\subseteq \mathbb N$, and let $\mathcal P_S$ denote the set of partitions whose parts lie in $S$; we consider $\emptyset \in \mathcal P_S$ for all $S\subseteq \mathbb N$. 
For $\lambda \in \mathcal P$, let $|\lambda|\geq 0$ denote the {\it size} (sum of parts), let $\ell(\lambda)\geq 0$ denote the {\it length} (number of parts), 
and let $m_i=m_i(\lambda)\geq 0$ be the {\it multiplicity} (frequency) of $i\in \mathbb N$ as a part of partition $\lambda$.

An important method in partition theory is the use of product-sum generating functions. For example, if $S\subseteq \mathbb N$, taking $z, q \in \mathbb C$ such that both $|q|<1$ and $|q^{\operatorname{min}S}|<  |z^{-1}|$, with $\operatorname{min}S$ denoting the least element of  subset $S$, 
then standard  generating function arguments  \cite{And, Fine} give 
\begin{equation}\label{partgen}
\prod_{n\in S}(1-z q^n)^{-1}\  =\  \sum_{\lambda\in \mathcal P_S}z^{\ell(\lambda)}q^{|\lambda|},
\end{equation}
where the right-hand sum is taken over partitions whose parts all lie  in $S$. 
Here we look at analogous generating functions for {\it ordered} sums of natural numbers.

Let $\mathcal C$ denote the set of {\it integer compositions}, which are {ordered} finite sums of natural numbers. We extend the partition-theoretic terminology and notations defined above to compositions, with the same meanings, e.g. for $c\in\mathcal C$ we let $|c|$ denote the sum of the parts, $\ell(c)$ denote the number of parts, etc.  
Let $\mathcal C_S$ denote compositions whose parts all lie in $S\subseteq \mathbb N$, thus $\mathcal C = \mathcal C_{\mathbb N}$; take the empty composition $\emptyset\in \mathcal C_S$ for all $S$. 
For $z, q \in \mathbb C$,  we define an auxiliary series
\begin{equation}\label{fdef} \phi_S(z; q):=1-z \sum_{n\in S}q^n,\end{equation} 
which converges when $|q|<1$ by comparison with geometric series;  note $\phi_S(0;q)=1$ identically. 
We view $\phi_S(z;q)$ as a composition-theoretic  analogue of the generating function $\prod_{n\in S}(1-zq^n)$; its reciprocal gives the following identity, an analogue of  \eqref{partgen} for compositions in $\mathcal C_S$. 

\begin{proposition}\label{prop1}  
For $S\subseteq \mathbb N$, $z,q\in \mathbb C$ such that $|q|<\frac{1}{1+|z|}$,  
we have \begin{equation*}
\frac{1}{\phi_S(z; q)}\  =\  \sum_{c\in\mathcal C_S} z^{\ell(c)}q^{|c|},
\end{equation*}
where the right-hand sum is taken over the set of compositions whose parts all lie in $S$.
\end{proposition}
%
%
%

\begin{remark}
The restriction on $|q|$ in Proposition \ref{prop1} is sufficient for convergence of the series on the right-hand side for all $z$ and $S$. 
For fixed $z, S$, the necessary condition is $|\sum_{n\in S}q^n|<  |z^{-1}|$.
\end{remark}

We postpone the proof of the above proposition, and all other proofs, for Section \ref{section2}. 


One might anticipate such analogies between partitions and compositions, since both represent sums of natural numbers.  
On the other hand, the unordered versus ordered arrangement of parts is a critical enumerative 
distinction (see the twelve-fold way in \cite[p. 71 ff.]{Stanley}).  It is well known that there are 
$2^{n-1}$  compositions of size $n$ \cite[p. 151]{compositions}, while the number $p(n)$ of 
partitions of size $n$ is  much smaller, only $O\left(e^{\pi\sqrt{2n/3}}/n\right)$ by the 
Hardy--Ramanujan asymptotic~\cite[p. 79]{HR}. 
The multinomial theorem \cite{multinomial} gives an 
explicit connection: for each partition $\lambda\in \mathcal P$, there are 
\begin{equation}\label{multi} \frac{\ell(\lambda)!}{m_1(\lambda)!\  m_2(\lambda) !\  m_3(\lambda)! \cdots}\  \geq\  1 \end{equation}
multiset permutations of the parts of $\lambda$, i.e., distinct compositions having  the same parts as $\lambda$. 

\subsection{Sums over partitions vs. sums over compositions}
Summations $\sum_{\lambda \in \mathcal P'}$ indexed by partitions $\lambda\in \mathcal P'\subseteq \mathcal P$, such as on the right-hand side of \eqref{partgen}, date back 
at least to work of MacMahon~\cite[p. 61ff.]{MacMahon} and Fine \cite[\S22]{Fine}. 
Are there natural examples of {\it composition}-theoretic series? 
In fact, many partition-theoretic series can also be expressed in terms of compositions. 

\begin{proposition}\label{symm}
If $g \colon \mathcal P \to \mathbb C$ is a function symmetric on the parts of $\lambda\in \mathcal P$ (i.e., unaffected by their order) and $\sum_{\lambda\in \mathcal P}g(\lambda)$ converges absolutely,  then 
\begin{equation*}
\sum_{\substack{\lambda\in\mathcal{P}}} g(\lambda)\  =\  \sum_{\substack{c \in\mathcal{C}}} \widehat{g}(c),
\end{equation*}
where $\widehat{g}(c):=g(\lambda) \cdot \frac{m_1(\lambda)! m_2(\lambda)! \cdots m_n(\lambda)!}{\ell(\lambda)!}$ for each $c\in \mathcal C$ having the same multiset of parts as $\lambda\in \mathcal P$. 
\end{proposition}

However, the equal sums in Proposition \ref{symm} are not necessarily equally ``natural''. In either the partition or composition setting, the summation may have a simple combinatorial or generating function interpretation, whereas in the alternative setting, the coefficients %
require additional computation of lengths and multiplicities, or arguments about multiset permutations.\footnote{We note a more ``natural'' formula with fewer parameters is not necessarily more convenient for applications.}  
Connections between partition sums and modular forms  are ubiquitous in the  literature. For instance, the Dedekind eta function $q^{1/24}\prod_{n\geq 1}(1-q^n),   q:=e^{2\pi\operatorname{i}\tau}, \tau \in \mathbb H,$ 
is the prototype of a weight 1/2 modular form in $\tau$. 
Its reciprocal is essentially Euler's generating function formula for the partition function 
$p(n)$ \cite[p. 3 ff.]{And}, which enumerates the partitions of $n\geq0,\  p(0):=1$:   \begin{equation*}
\frac{1}{(q;q)_{\infty}}\  =\   \sum_{\lambda \in \mathcal P}q^{|\lambda|}\  =\  \sum_{n\geq 0}p(n)q^n,
\end{equation*}  
where $(z;q)_{\infty}:=\prod_{n=0}^{\infty}(1-zq^n), z\in \mathbb C, |q|<1,$ denotes the $q$-{\it Pochhammer symbol}. 
If $z\neq 0$, the {\it Jacobi triple product formula} \cite{Jacobi} (see also~\cite[p. 21, Theorem 2.8]{And}), 
\begin{equation}\label{jtp}
(q;q)_{\infty}\cdot (-z^{-1};q)_{\infty}\cdot  (-zq;q)_{\infty}\ =\  \sum_{n=-\infty}^{\infty}z^n q^{\frac{n(n+1)}{2}},
\end{equation}
gives rise to numerous modular and partition-theoretic identities \cite{Ono_web}. More recently, deep connections have been drawn between sums over partitions and modular forms in connection with the $q$-bracket of Bloch and Okounkov, an expected value operator from statistical physics that induces modularity from certain partition-theoretic $q$-series  (see e.g. \cite{BO, BOW, padic, Robert_jtp, JWvI, Zagier}). 


%

Are there modular forms that arise  naturally as {\it composition}-theoretic $q$-series? We can answer this question in the affirmative by giving examples. Recall the following cases of \eqref{jtp} studied by Ramanujan, examples of what are now called {\it Ramanujan's theta functions} \cite{Berndt}:
\begin{flalign}\label{psi}  \psi(q) &:= \sum_{n\geq 0} q^{n(n+1)/2}\  =\  \prod_{k=1}^{\infty}(1-q^{k})(1+q^k)^{2},\\ \label{phi}
\varphi(q)\  &:=\  \sum_{n=-\infty}^{\infty}q^{n^2}\  =\  \prod_{k=1}^{\infty}(1-q^{2k})(1+q^{2k-1})^2
 .\end{flalign}
Up to multiplicative factors in $q=e^{2\pi\operatorname{i}\tau},$ these are modular forms of weight $1/2$ on $\text{SL}_2(\mathbb Z)$. 
Inspecting the summations in the above equations, it is noteworthy that the exponents of $q$ are {\it polygonal numbers}, that is, $n(n+1)/2$ is the $n$th {triangular} number, and $n^2$ is the $n$th perfect square, 
with $n\geq 1$. 
For $k\geq 3$, let  \begin{equation*} P_k := \left\{ 
 n\big( (k-2)n \pm (k-4)  \big)/2 : n\in\mathbb{N} \right\}
\end{equation*} denote the set of
 \emph{positive extended $k$-sided polygonal numbers}, which we will simply refer to
as the \emph{$k$-gonal numbers}. Thus $P_3$ is the set of triangular numbers, $P_4$ is the squares, and so on. 

Then one can write the two theta functions above in the form $\phi_S(z; q)$ as defined in \eqref{fdef}: 
\begin{equation*}
\psi(q) =\  1+\sum_{n\in P_3}q^n\  =\  \phi_{P_3}(-1,q),  \  \  \  \  \  \  \  \  \varphi(q)  =\  1+2\sum_{n\in P_4}q^n\  =\  \phi_{P_4}(-2,q).
\end{equation*} 
By Proposition \ref{prop1}, the reciprocals of these functions have natural representations as $q$-series summed over compositions whose parts are $k$-gonal numbers:
\begin{equation}\label{theta}
\frac{1}{\psi(q)}\  =\  \sum_{c\in\mathcal C_{P_3}} (-1)^{\ell(c)}q^{|c|},\  \  \  \  \  \  \  \   \frac{1}{\varphi(q)}\  =\  \sum_{c\in\mathcal C_{P_4}} (-2)^{\ell(c)}q^{|c|},
\end{equation}
with $|q|<1/2$ in the first case, and $|q|<1/3$ in the second. Multiplying these functions by appropriate rational powers of $q=e^{2\pi\operatorname{i}\tau},$ yields natural examples of composition-theoretic modular forms of weight $-1/2$. While it is beyond the scope of this study to pursue the question, 
it is natural to wonder whether the composition setting might provide insight into modular forms theory, as do partitions in e.g. \cite{BO, JS, Robert_jtp, JWvI, Zagier}.  

\subsection{Statement of main results}\label{main}
 Here we enumerate certain classes of restricted partitions of
size $n$, as sums over \emph{compositions} whose parts are $k$-gonal numbers. We prove these identities in Section \ref{section2} below, utilizing reciprocals of Ramanujan's theta functions that can be interpreted both as partition generating functions, and as composition-theoretic series similar to \eqref{theta}. 
 

Let $p_S(n) := \#\{ \lambda\in\mathcal{P}_S: |\lambda| = n \}.$ 
 For  $k\geq 5$, let $S_k:= \{ n \in\mathbb{N} : n\equiv 0, \pm 1 \pmod{k - 2} \}$.
\begin{theorem} \label{eventhm}
 For   $k\geq 6$ an even integer, the number  of  partitions of size $n$ whose parts are in $S_k$ is given by 
  \[ p_{S_k}(n) = (-1)^n 
 \sum_{\substack{c\in\mathcal{C}_{P_k} \\ |c| = n}   } (-1)^{\ell(c)}, \]
where the right-hand sum is taken over the set of size-$n$ compositions into $k$-gonal parts. \end{theorem}

For odd values of $k\geq 5$, a similar theorem exists, but it involves a {modified length statistic}. 
Define the subset $P^*_k\subset P_k$ as follows:
\[ P^*_k := \left\{  n\big( (k-2)n - (k-4) \big)/2 :
 n\in\mathbb{Z}, n\neq 0, \mbox{ and } n\equiv 0,1 \hspace{-3mm} \pmod{k-2} \right\}. \]
Let $\ell_{k}^*(c)$ denote the number of parts in composition $c$ that
are in $P^*_k$.
\begin{theorem} \label{oddthm}
  For $k\geq 5$ an odd integer, the number   of  partitions of size $n$ whose parts are in $S_k$ is given by 
   \[ p_{S_k}(n) = (-1)^n
    \sum_{\substack{c\in\mathcal{C}_{P_k} \\ |c| = n}   } (-1)^{\ell_{k}^*(c)}, \]
where the right-hand sum is taken over the set of size-$n$ compositions into $k$-gonal parts. \end{theorem}

%
%

The cases $k=3, 4$ are not covered by the above theorems, but do give interesting results. 
%
Let  $\operatorname{pod}(n)$ denote the number of {\it odd-distinct partitions} wherein odd parts may not
be repeated, with size equal to $n$; the function $\operatorname{pod}(n)$ was studied in detail by Hirschhorn~\cite[Chapter 32]{MDH} and 
Hirschhorn--Sellers~\cite{HS}. In Section \ref{section2}, we prove the following identity using Theorem \ref{eventhm}.

\begin{corollary}\label{podthm}
The number of odd-distinct partitions of size $n$   is given by 
 \begin{equation*} \label{podformula} \Pod(n)\  =\  (-1)^n{\sum_{\substack{c\in\mathcal{C}_{P_3}\\|c|=n} }} (-1)^{\ell(c)},
\end{equation*}
where the right-hand sum is taken over the set of size-$n$ compositions into triangular parts.
\end{corollary}

We recall that an \emph{overpartition} is an unordered sum of
positive integers wherein the first occurrence of each integer appearing as a part may or 
may not 
be overlined (see ~\cite{CL}).  
Let $\overline{p}(n)$ denote the {number of 
overpartitions} of size $n$.

\begin{theorem}\label{simpleop}
The number of overpartitions of size $n$ is given by 
 \begin{equation*} \overline{p}(n)\  =\  (-1)^n{\sum_{\substack{c\in\mathcal{C}_{P_4}\\|c|=n} }} (-2)^{\ell(c)}, 
\end{equation*}
where the right-hand sum is taken over the set of size-$n$ compositions into square parts. \end{theorem}

%
%
%
%
%

%
%

We note that the set $S_5$ is equal to $\mathbb{N}$, since every positive integer is congruent to $0, 1$ or $-1$ modulo 3, 
so $p_{S_5}(n) = p(n)$.   Thus, the $k=5$ case of Theorem~\ref{oddthm} yields the
following corollary.

\begin{corollary} \label{pofn} The number  of unrestricted partitions of size $n$ is given by
\[ p(n) = (-1)^n \sum_{\substack{c\in\mathcal{C}_{P_5}\\ |c|=n } }(-1)^{\ell_{5}^*(c) }, \] 
where the right-hand sum is taken over the set of size-$n$ compositions into pentagonal parts.\end{corollary}

Another formula for $p(n)$ that is similar in style, but does not follow directly
from Theorem~\eqref{oddthm}, is as follows. We require another modified length statistic. 
Let $\widehat{P}:= \{ k(3k\pm 1)/2: k\in\mathbb{N}\text{ and } k \text{ is even} \}.$
Let $\widehat{\ell}(c)$ denote the number of parts of composition 
$c$ that lie in $\widehat{P}$.
\begin{theorem} \label{pofn2}
The number of unrestricted partitions of size $n$ is given by 
 \[ p(n) = {\sum_{\substack{c\in\mathcal{C}_{P_5} \\ |c| = n}}}
 (-1)^{\widehat{\ell}(c)}, \]
 where the right-hand sum is taken over the set of size-$n$ compositions into pentagonal parts.
\end{theorem}

\begin{remark}  Using a theorem of Meinardus~\cite{GM54}, one can show that for $k\geq 5$, 
\begin{equation} p_{S_k}(n) \sim \frac{1}{8n} \csc\left( \frac{\pi}{k-2} \right) \exp\left( \pi \sqrt{\frac{2n}{k-2}}  \right)
\mbox{ as   } n\to\infty.
 \end{equation}
\end{remark}
%

\section{Proofs of identities from the preceding section} \label{section2}

\subsection{Central lemma and useful notations}
The following  lemma was proved by Salem~\cite{Salem}
in a different form, and re-proved by the authors in~\cite{SS_comb1} using the multinomial theorem. 
\begin{lemma}\label{lemma1}
For $a_i\in\mathbb C, a_0\neq 0$, let $g(q):=\sum_{n\geq 0}a_n q^n$ be analytic on $\{q\in \mathbb C : |q|<1\}$, and set $G(q):=1/g(q)$. Then on the domain of analyticity of $G(q)$, we have \[G(q)\  =\  \sum_{n=0}^{\infty}b_n q^n, \]
with the coefficient $b_n$ given as a sum over compositions $c\in\mathcal C$ of size $n$:\begin{equation*} \label{recipcoeff}
b_n\  =\  
a_0^{-1}{\sum_{\substack{c\in\mathcal{C}\\|c| = n}}} 
\left(-\frac{a_1}{a_0}\right)^{m_1 }
\left(-\frac{a_2}{a_0}\right)^{m_2 } 
\left(-\frac{a_3}{a_0}\right)^{m_3 }  \cdots 
\left(-\frac{a_n}{a_0}\right)^{m_n} .
\end{equation*}
\end{lemma}

\begin{remark}We refer the reader to the aforementioned references for proofs of this lemma.  \end{remark}

We employ Ramanujan's theta functions, together with Lemma \ref{lemma1}, to prove the theorems and corollaries. For $a,b \in \mathbb C$ such that $| a b | < 1$, let 
\begin{equation}\label{Rtheta} f(a,b) := \sum_{n=-\infty}^\infty a^{n(n+1)/2} b^{n(n-1)/2} =
\prod_{j=0}^\infty (1 + a^{j+1} b^{j}) (1 + a^{j} b^{j+1}) (1 - (a b)^{j+1}) \end{equation}
denote the general form of the \emph{Ramanujan theta function}~\cite[p. 6, Eq. (1.2.1)]{Berndt};
 the second equality follows
from the Jacobi triple product identity~\cite[p. 10, Eq. (1.3.10)]{Berndt}. 
We employ the usual abbreviation
%
  \[ (z_1, z_2, \dots, z_r; q)_\infty :=  \prod_{j=1}^r (z_j; q)_\infty, \]
with $z_j\in \mathbb C$.   Then one may re-write \eqref{Rtheta} in the form 
  \[ f(a,b) = (-a, -b, ab; ab)_\infty. \]
We note that $f(q,q^3)=\psi(q)$ from \eqref{psi}, and $f(q,q)=\varphi(q)$ from \eqref{phi}, with $|q|<1$.
 The Ramanujan theta functions of interest to us will be $f(\pm q, \pm q^{k-3})$, with the $\pm$ signs not necessarily matching. We note that, by consideration of products of Dedekind eta functions, 
     \[ q^{(k-4)^2/ (8k - 16)} f( \pm q,  \pm q^{k-3}) \] is a modular form of weight $1/2$. Under appropriate changes of variables, Ramanujan's theta functions are
equivalent to Jacobi theta functions; see the discussion in~\cite[p. 37 ff.]{S17}.
  
  \subsection{Proofs of Propositions \ref{prop1} and \ref{symm}}
\begin{proof}[Proof of Proposition \ref{prop1}] In Lemma \ref{lemma1}, set $a_0=1$, and  for $k\geq 1$, set $a_k=q^k$ if $k\in S$ and $a_k=0$ otherwise; then reorganize  $G(q)=\sum_{n\geq 0}b_n q^n$ as a sum over compositions in  $\mathcal C_S$. The proof of Lemma \ref{lemma1} in  \cite{SS_comb1}  interprets $G(q)=1/\phi_S(z; q)$ as a geometric series; consequentially, $|q|<1$ and $|\sum_{n\in S}q^n|<|z^{-1}|$ are necessary conditions for convergence on the right-hand side of the proposition. 
Recalling $\phi_S(0;q)=1$ identically, for $z\neq 0$ take $|q|<(1+|z|)^{-1}<1$.  Then by the triangle inequality, $|\sum_{n\in S}q^n| \leq \sum_{n\in S}|q|^n \leq \sum_{n\geq 1}|q|^n < \sum_{n\geq 1}(1+|z|)^{-n} =|z^{-1}|$, which proves the sufficient condition for convergence in $q$. We note this condition is necessary when $z=1, S=\mathbb N$, since $\sum_{c\in\mathcal C}q^{|c|}=1+\sum_{n\geq 1}2^{n-1}q^n=\frac{1-q}{1-2q}$ is convergent when $|q|< 1/2$.
\end{proof}

%
%

\begin{proof}[Proof of Proposition \ref{symm}] The identity follows directly from \eqref{multi} and the hypotheses; for partition $\lambda$, there are $\frac{\ell(\lambda)!}{m_1(\lambda)! m_2(\lambda)! \cdots m_n(\lambda)!}$ identical terms $\widehat{g}(c)$ on the right  side, whose sum is $g(\lambda)$. 
\end{proof}

\subsection{Proof of Theorem~\ref{eventhm}}
For $k\geq 6$ and $k$ even, 
  \[ f(q,q^{k-3}) = \sum_{j=-\infty}^\infty q^{j((k-2)j - (k-4))/2 }  = 1 + \sum_{n\in P_k}  q^n. \]
Our argument requires $|q|<1/2$ by Proposition \ref{prop1}.   By Lemma~\ref{lemma1},
\begin{equation} \label{rs1}  \frac{1}{f(q,q^{k-3})} = \sum_{n=0}^\infty q^n 
\sum_{\substack{c\in\mathcal{C}_{P_k} \\ |c|=n}}(-1)^{\ell(c)}.
\end{equation}
Now, by~\cite[p. 4, Eq. (1.2.3)]{And}, 
\begin{equation} \label{cs11}
  \frac{1}{f(-q,-q^{k-3})} =  \frac{1}{(q, q^{k-3}, q^{k-2}; q^{k-2})_\infty} =
  \sum_{n=0}^\infty p_{S_k}(n) q^n.
\end{equation}
Replace $q$ by $-q$ in~\eqref{cs11} to obtain
\begin{equation} \label{cs1}
  \frac{1}{f(q,q^{k-3})} =  \frac{1}{(-q, -q^{k-3}, q^{k-2}; q^{k-2})_\infty} =
  \sum_{n=0}^\infty (-1)^n  p_{S_k}(n) q^n.
\end{equation}
  Comparing coefficients of $q^n$ in~\eqref{rs1} and~\eqref{cs1}, yields
Theorem~\ref{eventhm}.  \hfill $\square$

\subsection{Proof of Theorem~\ref{oddthm}} 
 For $k\geq 5$ and $k$ odd, 
  \[ f(q,-q^{k-3}) = \sum_{j=-\infty}^\infty (-1)^j  (-q)^{j((k-2)j - (k-4))/2 }  
  = 1 + \sum_{n\in P^*_k}  q^n - \sum_{n\in P_k\setminus P^*_k} q^n. \]
A similar triangle inequality application to that  in the proof of Proposition \ref{prop1}, gives that $|q|<1/2$ will  suffice for convergence in our arguments in this case. Then by Lemma~\ref{lemma1},
\begin{flalign} \label{rs2}  \frac{1}{f(q,-q^{k-3})} &=
1 + \sum_{n=1}^\infty q^n 
\sum_{\substack{c\in\mathcal{C}_{P_k} \\ |c| = n }}  
\left( \prod_{i\in P_k\setminus P^*_k} 1^{m_i} \right)
\left( \prod_{j\in P^*_k} (-1)^{m_j} \right)
\\ \nonumber &= \sum_{n=0}^\infty q^n 
\sum_{\substack{c\in\mathcal{C}_{P_k} \\ |c|=n}} (-1)^{\ell_{k}^*(c)}.
\end{flalign}
Also, by~\cite[p. 4, Eq. (1.2.3)]{And},
\begin{equation} \label{cs12}
  \frac{1}{f(-q, -q^{k-3})} =  \frac{1}{(q, q^{k-3}, q^{k-2}; q^{k-2})_\infty} =
  \sum_{n=0}^\infty  p_{S_k}(n) q^n.  
\end{equation}
Replacing $q$ by $-q$ in~\eqref{cs12} gives
\begin{equation} \label{cs2}
  \frac{1}{f(q, -q^{k-3})} =  \frac{1}{(-q, q^{k-3}, -q^{k-2}; -q^{k-2})_\infty} =
  \sum_{n=0}^\infty (-1)^n  p_{S_k}(n) q^n.
\end{equation}
  Comparing coefficients of $q^n$ in~\eqref{rs2} and~\eqref{cs2}, yields
Theorem~\ref{oddthm}.    \hfill $\square$
%
%
%

\subsection{Proof of Theorem~\ref{simpleop}} 
Recall that 
\[ f(q,q) = \sum_{j=-\infty}^\infty q^{j^2} = 1 + 2\sum_{n\in{P_4}} q^n. \]
Take  $|q|<1$ such that $|\sum_{n\in{P_4}} q^n|<1/2$.  By Lemma~\ref{lemma1},
\begin{equation} \label{ramphi}
  \frac{1}{f(q,q)} = 1 + \sum_{n=1}^\infty q^n \sum_{\substack{ c\in\mathcal{C}_{P_4} \\
  |c| = n}  }
  (-2)^{\ell(c) }.
\end{equation}  
  By~\cite[p. 1623, Eq. (1.1)]{CL}, 
  \begin{equation} \label{opgf}
     \frac{(-q;q)_\infty}{(q;q)_\infty} = \frac{1}{f(-q,-q)} = \sum_{n=0}^\infty \overline{p}(n) q^n.
  \end{equation}
Replace $q$ by $-q$ in~\eqref{opgf} to obtain
\begin{equation}  \label{mopgf}
   \frac{1}{f(q,q)} = \sum_{n=0}^\infty (-1)^n\overline{p}(n) q^n.
\end{equation}
Comparing coefficients of $q^n$ in~\eqref{ramphi} and~\eqref{mopgf} yields
Theorem~\ref{simpleop}. \hfill $\square$

\subsection{Proof of Theorem \ref{pofn2}}
The theorem follows from Lemma~\ref{lemma1} in the same way as
the others, but by taking the reciprocal of the series $f(-q, -q^2)$. \hfill $\square$

\subsection{Proofs of Corollaries \ref{podthm} and \ref{pofn}}
 Corollary~\ref{podthm} is the $k=6$ case of Theorem~\ref{eventhm}, noting the set $P_3$ of triangular numbers is well known to equal the set $P_6$ of extended hexagonal numbers. 
Corollary~\ref{pofn} is the $k=5$ case of Theorem~\ref{oddthm}.  \hfill $\square$

\section{Further combinatorial applications of Lemma~\ref{lemma1}}  \label{section3}
\subsection{More on reciprocals of power series}
The key elements of the proofs in the preceding section are as follows:  We start with a lacunary series where the nonzero 
coefficients are all equal to $1$ (or $\pm 1$ with the signs in a predictable 
pattern) and where the nonzero coefficients have exponents that belong to
an easily identified set.  Then interpret the reciprocal series as a generating
function for partitions with certain restrictions on the parts. 

\begin{remark}For the sake of this section, we proceed in the spirit of formal power series, although in most cases $|q|<1$ suffices for convergence; previous convergence criteria  continue to hold.\end{remark}
 One could find further formal power series that have similar characteristics.  For example,
if $\alpha < \beta$ are positive integers of the same parity, then $f(q^\alpha, q^\beta)$
is a lacunary series with nonzero coefficients all equal to $1$, and occurring at
powers of $q$ in the set
\[ R_{\alpha,\beta} := \left\{ j \Big(  (\alpha+\beta)j + (\alpha-\beta)\Big)/2 : j \in \mathbb{Z} \right\}. \]
We note that $1/f( -q^\alpha, -q^\beta)$ is the generating function for $p_{T_{\alpha,\beta} }(n)$
with  $T_{\alpha,\beta} \subset \mathbb N$ defined by  \[ T_{\alpha,\beta} := \{ n\in\mathbb{N} : n \equiv 0, \pm\alpha \hskip -3mm\pmod{\alpha+\beta} \}. \]
Using the methods above, and setting $R_{\alpha,\beta}^* := R_{\alpha,\beta} \setminus \{ 0 \}$, we conclude that
\begin{equation}
  p_{T_{\alpha,\beta} }(n) = (-1)^n \sum_{ \substack{ c\in\mathcal{C}_{R_{\alpha,\beta}^*} \\ |c| = n   }  }
  (-1)^{\ell(c)}.
\end{equation}
 
  We can apply the same technique to a lacunary series where the 
 nonzero coefficients differ from $\pm 1$, but the resulting sum over compositions will be more complicated.   For instance the Jacobi identity~\cite[p. 14, Thm. 1.3.9]{Berndt} is
 \begin{equation} \label{JI}
  \prod_{j=1}^\infty (1-q^j)^3 = \sum_{n=0}^\infty (-1)^n (2n+1) q^{n(n+1)/2}.
 \end{equation} 
 Conveniently, the reciprocal of~\eqref{JI} is the generating function for 
 the enumeration function of a well-known combinatorial object, namely $p^{(3)}(n)$, the
 number of three-colored partitions of $n$:
 \begin{equation} \label{p3nGF}
    \sum_{n=0}^\infty p^{(3)}(n) q^n = \prod_{j=1}^\infty \frac{1}{(1-q^j)^3}.
 \end{equation}
Thus by Lemma~\ref{lemma1}, 
\begin{equation} \label {p3formula}
   p^{(3)}(n) = \sum_{\substack{c \in \mathcal C_{P_3} \\ |c| = n}} 
   3^{m_1} (-5)^{m_3} 7^{m_6} (-9)^{m_{10}} 11^{m_{15}} \cdots 
   =  \sum_{\substack{c \in \mathcal C_{P_3} \\ |c| = n}} 
  (-1)^{\ell^+(c) } \prod_{j\geq 1} (2j+1)^{m_{j(j+1)/2}} ,
\end{equation}
where $\ell^+(c)$ is the total number of parts of $c$ in the set
$\{3,10, 21, 36, 55, 78, 105, 136, \dots \} = \{ n(2n+1) : n\in\mathbb{N} \}$.

While~\eqref{p3formula} is perhaps less elegant than the results above due to the
more complicated summands, the presence of factors greater than $1$ raised to positive
powers accelerates the efficiency of the formula.   For example,
consider $p^{(3)}(5) = 108$.
The total $108$ can be found by summing over the compositions $1+1+1+1+1$, a contribution
of $3^5$, and the three compositions associated with the partition $3+1+1$, a contribution of 
$3^2 (-5)$, yielding $3^5 + 3(3^2)(-5) = 108$.  
This calculation is much quicker than directly enumerating the  
three-color partitions of $5$.  
\begin{remark}By  Meinardus~\cite{GM54}, the asymptotic growth rate is
\begin{equation} \label{p3asy}
  p^{(3)}(n) \sim \frac{1}{8\sqrt{2}\ n^{3/2}} e^{\pi \sqrt{2n}} \mbox{   as   } n\to\infty.
\end{equation}
\end{remark}

Encouraged by the potential for efficient computations, we look for additional lacunary series to which
our method can conveniently be applied.
In Ono and Robins~\cite[p. 1022, Eqs. (7), (9)]{OR}, one finds identities 
equivalent to the following:
\begin{align}
   \frac{(q;q)_\infty^5}{(q^2;q^2)_\infty^2} &= \sum_{j=-\infty}^\infty (1-6j) q^{j(3j-1)/2}
   \label{or7}, \\
   \frac{ (q;q)_\infty^2 (q^4;q^4)_\infty^2}{(q^2;q^2)_\infty} 
   &= \sum_{j=-\infty}^\infty (3j+1) q^{j(3j+2)} \label{or9}.
\end{align}
We can interpret the reciprocals of~\eqref{or7} and~\eqref{or9}, respectively, as the generating functions
\begin{align}
  \frac{(q^2;q^2)_\infty^2 }{(q;q)_\infty^5} = 
  \frac{(-q;q)_\infty^2}{(q;q)_\infty^3} &= \sum_{n=0}^\infty r(n) q^n, \\
  \frac{ (q^2;q^2)_\infty }{ (q;q)_\infty^2 (q^4;q^4)_\infty^2}  = 
  \frac{(-q;q)_\infty}{ (q;q)_\infty (q^4;q^4)_\infty^2} & = \sum_{n=0}^\infty s(n) q^n,
\end{align}
 where $r(n)$ denotes the number of three-colored (say, red, yellow, blue) overpartitions
 of size $n$ wherein no overlined blue parts appear; and
 $s(n)$ denotes the number of three-colored overpartitions of size $n$
 wherein no overlined yellow or blue parts appear, and all yellow and blue parts are
 multiples of $4$.  
 Applying Lemma~\ref{lemma1} and  comparing coefficients of $q^n$ yields 
 \begin{align}
 r(n) &= \sum_{ \substack{ c\in\mathcal{C}_{P_5} \\ |c| = n  } }
 5^{m_1} (-7)^{m_2} 11^{m_5} (-13)^{m_7} 17^{m_{12}} (-19)^{m_{15}}\cdots
 = \sum_{ \substack{ c\in\mathcal{C}_{P_5} \\ |c| = n  } }
    \prod_{j\in\mathbb{Z^*}} (6j-1)^{m_{j(3j-1)/2}}, \\
 s(n) &=  \sum_{ \substack{ c\in\mathcal{C}_{U} \\ |c| = n  } }
 2^{m_1} (-4)^{m_5} 5^{m_8} (-7)^{m_{16}}  8^{m_{21}} (-10)^{m_{33}} \cdots
  =  \sum_{ \substack{ c\in\mathcal{C}_{U} \\ |c| = n  } }
 \prod_{j\in\mathbb Z^*} (-1-3j)^{m_{j(3j+2)}}, 
 \end{align}
 where $\mathbb Z^*$ denotes the set of nonzero integers, and  we let  $U:= \{ j(3j+2) : j\in \mathbb{Z}^*
\}.$
 
 \begin{remark} Again, Meinardus' theorem~\cite{GM54} gives us the asymptotic growth rates as $n\to \infty$:
 \begin{align}
  r(n) &\sim  \frac{  1 }{12\sqrt{2}\ n^{3/2} } e^{2\pi\sqrt{2n/3}}, \\
  s(n) &\sim \frac{  1 }{6\ n^{3/2}} e^{2\pi\sqrt{n/3}}.
   \end{align}\end{remark}
   
\subsection{Ratios of power series}
The first Rogers--Ramanujan identity~\cite[p. 328 (2)]{R94} in its analytic form is given by
\begin{equation} \label{rr1}
 \sum_{n=0}^\infty \frac{q^{n^2}}{(1-q)(1-q^2)\cdots(1-q^n)} = \frac{1}{(q,q^4;q^5)_\infty}
\end{equation}  for $|q|<1$.
Interpreting the left and right sides of~\eqref{rr1} as generating functions, 
MacMahon~\cite[p. 33]{MacMahon} obtained the following. 
\begin{theorem*}[First Rogers--Ramanujan identity, partition version]
  The number of parttions of size $n$ into parts differing by at least $2$ equals the number
  of partitions of size $n$ into parts $\equiv\pm 1\pmod{5}$.
\end{theorem*}
   
   As the Rogers--Ramanujan identities are among the most celebrated results in $q$-series,
it is tempting to see what can be said about them in our present context.  
Let us define $rr(n)$ by
\[  \sum_{n=0}^\infty rr(n) q^n = \frac{1}{(q, q^4;q^5)_\infty}. \]Thus  $rr(n)$ denotes the number
of partitions of size $n$ into parts $\equiv\pm 1\pmod{5}$, as well as the number of partitions of size $n$ into
parts that differ by at least two.  

Following our earlier procedure, since 
\begin{multline} 
(q,q^4;q^5)_\infty = 1 - q - q^4 + q^5 - q^6 + q^7 - q^9 + 2q^{10} - 2q^{11} + q^{12} +
  q^{13} - 2q^{14} + 3q^{15} - 3q^{16} + 2q^{17} - 3q^{19} \\+ 5q^{20} - 5q^{21}  + 3q^{22}
    + q^{23} - 5q^{24} + 7q^{25} - 7 q^{26} + 4q^{27} + q^{28} -7q^{29}+ 11q^{30} - \dots,
 \end{multline} we can apply Lemma~\ref{lemma1} and extract the coefficients of $q^n$ to find that
\begin{equation} \label{weirdrr}  rr(n) = \sum_{\substack{c\in A \\ |c| = n} }(-1)^{m_5 + m_7 + m_{12} + m_{13}} 
2^{m_{11}+m_{14}} (-2)^{m_{10} + m_{17} + m_{33}} 3^{m_{16}+m_{19}}
(-3)^{m_{15}+m_{22}} \cdots 
\end{equation}
for the set $A = \mathbb{N}\setminus\{ 2,3, 8, 18\}$; but the general summand on the right-hand side of~\eqref{weirdrr}
has no easily discernible pattern.   Accordingly, we proceed by reorganizing the infinite product:
\begin{align}
   \frac{1}{(q,q^4;q^5)_\infty} &= \frac{ f(-q^2,-q^3)}{(q;q)_\infty} = 
   \left( \sum_{j=-\infty}^\infty (-1)^j q^{j(5j-1)/2} \right) \left( \frac{1}{ \sum_{h=\infty}^\infty (-1)^h q^{h(3h-1)/2} } \right)
\\
\nonumber &= \left( \sum_{i=0}^\infty a_i q^i \right) \left( \sum_{h=0}^\infty q^h \sum_{ \substack{ c\in{\mathcal{C}_{P_5}} \\ 
|c| = h}  } (-1)^{\widehat{\ell}(c)}  \right) = \sum_{n=0}^\infty q^n \left( \sum_{i=0}^n a_i \sum_{ \substack{ c\in{\mathcal{C}_{P_5}} \\ |c| = n-i } }
(-1)^  {\widehat{\ell} (c)  } \right),
\end{align} where the last equality follows by taking the Cauchy product; and so
\begin{equation}
  rr(n) = \sum_{i=0}^n a_i \sum_{ \substack{ c\in \mathcal{C}_{P_5}\\ |c| = n-i  }  } (-1)^{ \widehat{\ell } (c )  },
\end{equation} where
\[  a_i =   \left\{ \begin{array}{rl} 1 & \mbox{if $i$ is of the form $10j^2\pm j$ for some nonnegative integer $j$,  }\\
                                                -1 & \mbox{if $i$ is of the form $10j^2\pm 9j +2$ for some nonnegative integer $j$,}
\\                                                0 & \mbox{otherwise.} \end{array}  \right. \]
\begin{remark}
As determined by Lehner~\cite[p. 655, Eq. (12.4)]{L41},
\[ rr(n) \sim \frac{ 1 }{  4 \sqrt[4]{15}  \sqrt{ \frac{ 5 - \sqrt{5}  }{ 8 }  }\  n^{3/4} } e^{ 2\pi \sqrt{n/15}  } 
 \mbox{  as  }
n\to\infty.\]\end{remark}

The first Rogers--Ramanujan identity was used as an illustrative example, but the preceding technique could
be used to find a sum-over-restricted-compositions expression for the coefficient of $q^n$ in the 
Maclaurin series expansion of any quotient of Ramanujan theta functions, such as appear in the identities of
Slater~\cite{S52}; see also~\cite[Appendix A]{S17}.
   
 \section{Composition-theoretic Dirichlet series} \label{section4}
 
Here we note connections of composition-theoretic series to partition zeta functions, partition-theoretic series analogous to the Riemann zeta function $\zeta(s):=\sum_{n\geq 1}1/n^s, \operatorname{Re}(s)>1,$ introduced by the first author in \cite{Robert_zeta}. Expanding in the additive-multiplicative direction, the first author in \cite{Robert_arithmetic, Robert_PhD} outlines a multiplicative theory of integer partitions parallel in many respects to major threads of elementary and analytic number theory, beginning with the definition of a multiplication operation between partitions, which is the multiset union of their parts.\footnote{See \cite{supernorm, OSW2, Robert_PhD, SS-zeta} for further reading about this additive-multiplicative theory.} Compositions do not admit the same theory, since they are order-dependent;  their multiplication operation is non-commutative, so a non-commutative analogue of number theory would arise.

In \cite{Robert_zeta, Robert_arithmetic, Robert_PhD, SS-norm}, a multiplicative statistic on integer partitions is introduced, the {\it norm} $N(\lambda)$, defined by $N(\emptyset):=1$, and for partition $\lambda=(\lambda_1, \lambda_2, \dots, \lambda_r)$, by the product of the parts: 
\begin{equation} N(\lambda) := \lambda_1 \lambda_2 \lambda_3 \cdots \lambda_r. 
\end{equation}
This statistic is invariant with respect to the order of the parts; we will extend the norm statistic to compositions as well, such that $N(c)$ denotes the product of the parts of composition $c$.

A {\it partition zeta function} $\zeta_{\mathcal P'}(s)$ is defined in \cite{Robert_zeta} by the following sum over partitions in the proper subset $\mathcal P' \subsetneq \mathcal P$:
\begin{equation}\label{pzetadef} \zeta_{\mathcal P'}(s):=\sum_{\lambda\in\mathcal P'} N(\lambda)^{-s},\end{equation}
with the domain $s\in\mathbb C$ depending on subset $\mathcal P'$. For convergence, there must be a restriction on the maximum number of $1$'s allowed as parts of a partition $\lambda \in \mathcal P'$, or else some  summands of \eqref{pzetadef} may have infinitely many copies (adjoining $1$'s to a partition does not change the norm). 

Partition zeta functions enjoy many nice identities analogous to classical zeta function theorems \cite{Robert_zeta}. Let $\mathbb N^*$ denote the {\it natural numbers strictly greater than $1$}; thus $\mathcal P_{\mathbb N^*}$ (resp. $\mathcal C_{\mathbb N^*}$) denotes the set of  partitons (resp. {compositions}) having no part equal to $1$. As a generalization of $\zeta(s)$, if $\mathcal P' = \mathcal P_{T^*}$ for $T^*\subseteq \mathbb N^*,$ then $\zeta_{\mathcal P_{T^*}}(s)$ has an Euler product: 
\begin{equation}\label{Euler1} \zeta_{\mathcal P_{T^*}}(s)=\prod_{n\in T^*}(1-n^{-s})^{-1},\end{equation}
with domain $\operatorname{Re}(s)>1$. We note that $\zeta(s)$ is the case $T^*=\mathbb P$.
%
%
%
%
%
%
%
%
%

Along similar lines, for $\mathcal C'\subsetneq \mathcal C$ such that the number of 1's cannot exceed a fixed maximum number $\geq 0$ (or else the series diverges to infinity), let us define a {\it composition zeta function}:
\begin{equation}\label{czetadef} \zeta_{\mathcal C'}(s):=\sum_{c\in\mathcal C'} N(c)^{-s},\end{equation}
with the domain $s\in\mathbb C$ depending in this case on  $\mathcal C'$.   These also enjoy nice identities. 
For instance, the following identity gives a composition-theoretic analogue of  \eqref{Euler1}. 
\begin{proposition}\label{prop3} For $T^*\subseteq \mathbb N^*$, $s\in\mathbb C$ such that $|\sum_{n\in T^*}n^{-s}|<1$, we have 
\begin{equation*}\label{czeta2}\zeta_{\mathcal C_{T^*}}(s)= \frac{1}{1-\sum_{n\in T^*}n^{-s}}.\end{equation*}
\end{proposition}
 
 \begin{proof}  We proceed formally, with convergence in $s$ dependent on choice of $T^*\subseteq \mathbb N^*$. In Lemma \ref{lemma1}, set $q=1 , a_0=1$, and  for $k\geq 1$, set $a_k= -k^{-s}$ if $k\in T^*$ and $a_k=0$ otherwise. Noting $G(1)$ is absolutely convergent, reorganize  $G(1)=\sum_{n\geq 0}b_n$ as a sum over compositions in  $\mathcal C_{T^*}$. 
  \end{proof}

%
%
%
%
%
%
For example, setting $T^*=\mathbb N^*$ in Proposition \ref{prop3}, then $\mathcal C_{T^*}=\mathcal C_{\mathbb N^*}$ (compositions with no $1$'s) and we arrive at a formula relating the composition zeta function $\zeta_{\mathcal C_{\mathbb N^*}}(s)$ to the classical case:
\begin{equation}\label{czeta2}\zeta_{\mathcal C_{\mathbb N^*}}(s)= \frac{1}{1-\left(\zeta(s)-1\right)}= \frac{1}{2-\zeta(s)}.\end{equation}

Let $h:\mathcal C \to \mathbb C$ be a composition-theoretic function. One can also define composition-theoretic Dirichlet series $\sum_{c\in \mathcal C'}h(c)N(c)^{-s}$ similar to the partition Dirichlet series in \cite{ORS, Robert_PhD, SS-zeta}, such as in the following identity resembling Proposition \ref{prop1}.

 \begin{proposition}\label{squares2} 
For $T^*\subseteq \mathbb N^*$,  $s,z \in\mathbb C$ such that $|\sum_{n\in T^*}n^{-s}|<|z^{-1}|$, we have 
\begin{equation*}
 \frac{1}{1-z\sum_{n\in T^*}n^{-s}}=\sum_{c\in\mathcal C_{T^*}}z^{\ell(c)}N(c)^{-s}. 
\end{equation*}
\end{proposition}

 \begin{proof} Proceed exactly as in the proof of Proposition \ref{prop3}, except set $a_k=-zk^{-s}$ if $k\in T^*$. 
  \end{proof}
 
Then we can prove identities analogous to those in Section \ref{main}, in multiplicative number theory. Recall the classical M\"{o}bius function $\mu(n)$ equals zero if $n\in \mathbb N$ is non-squarefree, and is equal to $(-1)^{\omega(n)}$ if $n$ is squarefree, where $\omega(n)$ is the number of different prime factors of  $n$.


\begin{theorem} \label{muthm}
The M\"{o}bius function $\mu(n)$ is given by the following sum over compositions with no part equal to 1, having norm equal to $n$:
 \[\mu(n) = {\sum_{\substack{c\in\mathcal{C}_{\mathbb N^*} \\ N(c) = n}}}
 (-1)^{\ell(c)}. \]
\end{theorem}
%
%

 \begin{proof}  Set $T^*=\mathbb N^*$ and $z=-1$ in Proposition \ref{squares2} to yield, for $\operatorname{Re}(s)>1$, 
 \begin{equation*}
  \frac{1}{\zeta(s)}\  =\  \sum_{c\in\mathcal C_{\mathbb N^*}}(-1)^{\ell(c)}N(c)^{-s}, \end{equation*}
 which can be rewritten as an equality between Dirichlet series:
 \begin{equation*}
\sum_{n\geq 1}\mu(n)n^{-s}\  =\   \sum_{n\geq 1}\left(\sum_{\substack{c\in\mathcal{C}_{\mathbb N^*}\\ N(c)=n }}\ (-1)^{\ell(c)}\right)n^{-s}. \end{equation*}
Comparing coefficients of $n^{-s}$ gives the theorem. 
   \end{proof}

%
%
%
%
%
%
%
%
%
%
%
%
\begin{remark} The theorem above is noteworthy because it contains a sum over compositions of fixed {norm},  instead of the usual size invariant. \end{remark}

 \section{Closing remarks and questions}
 We have proved a number of identities that show integer compositions to be a natural setting to express various kinds of analytic and combinatorial identities. While compositions are distinct  from partitions as combinatorial objects, due to their respecting the additional constraint of ordering, they turn out to have comparable features like generating functions, Dirichlet-like series, etc.; and we can use compositions to give new formulas for partition statistics, as well as for the classical M\"{o}bius function. 
On the other hand, compositions are less well-suited  to many techniques used in partition theory, such as relating Young diagrams to $q$-Pochhammer symbols $(z;q)_{\infty}$; and generating functions like Propositions \ref{prop1} and \ref{prop3} do not look as amenable to formal manipulation as $q$-Pochhammer-based partition generating functions, or Euler products. 


Are there further connections between composition generating functions and modular forms? We prove here that these functions can be related to Ramanujan's theta functions, which are essentially modular forms; further study in this direction seems warranted. Reciprocals of theta functions are connected to compositions into $k$-gonal parts; are there deeper connections between the symmetry groups of  regular $k$-gons, and the symmetries of modular forms? Are there interesting results related to a composition analogue of the $q$-bracket of Bloch--Okounkov? 

Do composition-theoretic Dirichlet series fit as nicely into composition theory as partition Dirichlet series do in partition theory \cite{ORS, Robert_zeta, SS-zeta}? On that note, does a nice {\it non-commutative} multiplicative theory of compositions exist as a composition analogue of multiplication in the integers, as we suggested in the previous section, along the lines of \cite{Robert_arithmetic}? Using the dictionary between partitions and arithmetic in \cite{supernorm},   the previous question also suggests a non-commutative theory of prime factorization in the integers. What other partition-theoretic and number-theoretic objects and structures admit  composition-theoretic analogues?
 \section*{Acknowledgments}
The authors are grateful to George Andrews, Maurice  Hendon, Mike Hirschhorn, Matthew Just,  Jeremy Lovejoy, Ken Ono, C\'{e}cile Piret and James Sellers for comments that benefited our work. In particular, we thank C. Piret for advice on proving convergence in Prop. \ref{prop1}.  

\end{document}